\documentclass[12pt,a4paper]{amsart}
\usepackage{amsmath,amsthm}
\usepackage{amssymb}
\usepackage[a4paper,centering,twoside,textwidth=6in]{geometry}
\usepackage[arrow,matrix]{xy}
\usepackage{amsmath,amsfonts,amssymb,amscd,verbatim,delarray,verbatim,epigraph}
\theoremstyle{plain}
\newtheorem{theorem}{Theorem}[section]

\newtheorem{corollary}[theorem]{Corollary}
\newtheorem{Proposition}[theorem]{Proposition}
\newtheorem{Lemma}[theorem]{Lemma}

\theoremstyle{definition}
\newtheorem{defn}[theorem]{Definition}

\newtheorem{remark}[theorem]{Remark}

\newcommand{\thmref}[1]{Theorem~\ref{#1}}

\newcommand{\lemref}[1]{Lemma~\ref{#1}}
\newcommand{\corref}[1]{Corollary~\ref{#1}}

\newcommand{\propref}[1]{Proposition~\ref{#1}}

\DeclareMathOperator{\PGL}{PGL} \DeclareMathOperator{\GL}{GL}
 
\newcommand{\BZ}{\mathbb{Z}}
\newcommand{\BC}{\mathbb{C}}

\newcommand{\BP}{\mathbb{P}}
\newcommand{\BF}{\mathbb{F}}\newcommand{\BQ}{\mathbb{Q}}
\newcommand{\BA}{\mathbb{A}}

\DeclareMathSymbol{\twoheadrightarrow}  {\mathrel}{AMSa}{"10}

\def\Q{{\mathbb Q}}

\def\A8{{\mathbf A}_8}
\def\Bir{\mathrm{Bir}}

\def\Aut{\mathrm{Aut}}
                                               \def\Bir{\mathrm{Bir}}

\def\GL{\mathrm{GL}}

\def\A{\mathcal{A}}

\def\dim{\mathrm{dim}}

       \def\PGL{\mathrm{PGL}}

          \def\l1{{\mathbf 1}}

                                                                                        \newcommand{\ov}{\overline}

\title[Jordan groups]{Jordan properties of automorphism groups  of certain open algebraic  varieties}
\thanks{The second named author is partially supported by a grant from the Simons Foundation (\#246625 to Yuri Zarkhin). Part of this work was done in May-June 2016 during his stay at the Max-Planck-Institut f\"ur Mathematik, whose hospitality and support are gratefully acknowledged.}

\author {Tatiana Bandman}

\author[Yuri G.\ Zarhin]{Yuri G.\ Zarhin}

\address{Department of
Mathematics, Bar-Ilan University, 5290002, Ramat Gan, ISRAEL}
\email{bandman@macs.biu.ac.il}
\address{Department of Mathematics, Pennsylvania State University,
University Park, PA 16802, USA}

\email{zarhin\char`\@math.psu.edu}

\begin{document}

\begin{abstract}
Let $W$ be a quasiprojective variety over an algebraically closed field of characteristic zero. Assume  that $W$   is birational to a product of a  smooth projective   variety  $A$ and the projective line. We prove that  if  $A$  contains no rational curves   then the automorphism  group   $G:=\Aut(W)$
of $W$  is Jordan.   That means that   there is a positive integer $J=J(W)$ such that every finite subgroup $\mathcal{B}$ of ${G}$ contains a   commutative subgroup 
$\mathcal{A}$  such that $\mathcal{A}$ is normal in   $\mathcal{B}$  and the index $[\mathcal{B}:\mathcal{A}] \le J$ .   \end{abstract}


\maketitle

\section{Introduction}\label{introduction}

Throughout this paper $k$ is an algebraically closed field of characteristic zero.
  All varieties, if not indicated otherwise, are irreducible, algebraic, and defined over $k.$ 
If $X$ is an algebraic variety over $k$ then we  write $\Aut(X)$ for its group of (biregular) automorphisms and $\Bir(X)$ for its group of birational automorphisms.
 As usual,  $\mathbb{P}^n$ stands for the $n$-dimensional  projective space and $\mathbb{A}^n$  (  $\mathbb{A}^n_{x_1,\dots,x_n}$)  for the $n$-dimensional affine space (with coordinates  $x_1,\dots,x_n$, respectively).

The definition of a  Jordan group  was    introduced in  \cite{Popov1}.
\begin{defn}\label{jordan}
 A group $\mathcal{G}$ is called {\bf Jordan} \cite{Popov1} if there exists a positive integer $J$ that enjoys the following property. Every finite subgroup $\mathcal{B}$ of $\mathcal{G}$ contains a commutative subgroup $\mathcal{A}$ such that $\mathcal{A}$ is normal in $\mathcal{B}$ and the index $[\mathcal{B}:\mathcal{A}] \le J$.  Such a smallest $J$ is called the Jordan index of $\mathcal{G}$ and denoted by $J_{\mathcal{G}}.$ \end{defn}

\begin{defn}\label{groupproperty} Let 
 $G$  be a group. \begin{itemize}\item[(a)]
$G$ is called {\sl bounded}  \cite{Popov2,PS1} if there is a positive integer 
$C=C_G$ such that the order of every finite subgroup of $G$  does not exceed $C$. 
\item[(b)]  $G$ is called {\sl quasi-bounded} if  there is a nonnegative integer  $a:=a(G)$ such that each finite abelian subgroup   of $G$ is generated by at most $A$ elements.
\item[(c)] $G$
is called {\sl strongly  Jordan} \cite{PS2,BZ16} if it is Jordan and quasi-bounded.
\end{itemize}\end{defn}

\begin{remark}
\label{exactR}
\begin{itemize}
\item[(i)]
If 
$$\{0\} \to G_1 \to G \to G_2\to \{0\}$$
is a short exact sequence of groups and both $G_1$ and $G_2$  are  bounded (resp. quasi-bounded) then one may easily check that
$G$ is also bounded (resp.  quasi-bounded).  
Indeed, let $H$ be a finite (resp. finite abelian) subgroup of $G$. Let $H_2$ be the image of $H$ in $G_2$ and $H_1$ the intersection of $H$ and
the kernel of $G \to G_2$. Then $H$ sits in the short exact sequence
$$\{0\} \to H_1 \to H \to H_2\to \{0\}$$
where $H_i$ is a finite (resp. finite abelian) subgroup of $G_i, \ i=1,2.$
If both $G_1$ and $G_2$ are bounded then the order of $H$ does not exceed $C_{G_1} C_{G_2}$, i.e., $G$ is also bounded.
If both $G_1$ and $G_2$ are quasi-bounded then $H$ is generated by, at most, $a(G_1)+a(G_2)$ elements \cite[Lemma 2.3]{MengZhang}. 
\item[(ii)]
If both $G_1$ and $G_2$ are Jordan then $G$ does {\sl not} have to be Jordan.
\item[(iii)] Clearly, every subgroup of a (strongly) Jordan group is also  (strongly) Jordan.\end{itemize}
\end{remark}

 The group  $\GL_{n}(\BZ)$ is   bounded by 
 Minkowski's Theorem (\cite[Sect. 9.1]{SerreFG}).
The classical theorem of Jordan (\cite[Sect. 36]{CR}, \cite[Sect. 9.2]{SerreFG},\cite{MuTu}) asserts that $\GL(n,k)$ is strongly Jordan. 
  An example of a non Jordan group is  given by  $\GL (n,  \overline \BF_p)$   where 
$\overline\BF_p$ is the algebraic closure of a finite field $\mathbb  F_p$ and $n \ge 2$.

We refer the reader to \cite{Popov2} for references and survey on this topic.

 Let  $X$ be an algebraic variety over $k.$     It is known that $\Aut(X)$ is Jordan if either $\dim(X)\le 2$  \cite{Popov1,BZ15} or $X$ is projective \cite{MengZhang}. 
It is also known (\cite{PS1}    combined with \cite{Bir}),  that   if $X$ is an  irreducible  variety  then $\Bir(X)$ is Jordan
 if   either $q(X)=0$ or  $X$ is  not uniruled     (in particular,  Cremona groups $\Bir(\BP^N)$  and groups $\Aut(\BA^N)$ are Jordan).

 On the other hand, $\Bir(X)$ is {\sl not} Jordan if $X$ is birational to a product $A\times \mathbb{P}^n$ where
$n \ge 1$ and $A$ is a positive-dimensional abelian variety over $k$ \cite{ZarhinPEMS}. 

 Since $\Aut(X)$ is a subgroup
 of $\Bir(X)$, it is Jordan whenever $\Bir(X)$ is Jordan. But $\Aut(X)$ may be Jordan when $\Bir(X)$ is not. 
 To the best of our knowledge, there is no example of an algebraic variety  with non-Jordan automorphisms group. 
The aim of this paper is to prove the  Jordan property of the group $\Aut(X)$  for open subsets of certain uniruled varieties. 
\begin{defn}\label{rigid} 
We call a smooth projective variety $A$  {\sl rigid} if it is  irreducible    and  contains no rational curves.\end{defn}
We prove the following 
\begin{theorem}
\label{main}
Let $W$ be an irreducible quasiprojective variety that is birational
 to a product  $A\times \mathbb{P}^1$ where $A$ is  a smooth rigid projective  variety.  
  Then $\Aut(W)$ is strongly Jordan.
\end{theorem}

The case of  $\dim(W)=2, \ \dim(A)=1$ was done in \cite{ Popov1,ZarhinTG15,BZ15}.

The case of  $\dim(W) =3$  was studied in    \cite{ZarhinPEMS,PS0,PS1,BZ16,PS2}.  Here is the final answer for $\Bir(X)$  \cite{PS2}. 
Let $X$ be a threefold. Then  $\Bir(X)$ is   not Jordan if and only if    either $X$ is birational to $E\times \BP^2,$   where $E$ is an elliptic curve,   or $X$ is birational 
to $S\times \BP^1,$
where $S$ is one of the following:
 
 {\bf Case  1.} An abelian surface;

{\bf  Case  2.}  A bielliptic surface;

 {\bf    Case  3.} A surface with Kodaira dimension  $ \varkappa(S)=1$ such that the 
Jacobian fibration of the pluricanonical fibration is locally trivial in Zariski topology.

Thus, \thmref{main} leads to the following 
\begin{corollary}\label{main3d} Assume that $W$ is a quasiprojective    irreducible variety of dimension $d\le 3.$ Assume that $ W$  is not  birational to $E\times \BP^2, $  where $E$ is an elliptic curve.  Then $\Aut (W)$ is Jordan.
\end{corollary}

\begin{remark}
\label{smooth}
Let $W$ be a (nonempty) irreducible algebraic variety over $k$ and $W^{\mathrm{ns}}\subset W$ the open dense (sub)set of its nonsingular points. Then $u(W^{\mathrm{ns}})\subset W^{\mathrm{ns}}$ for each $u\in\Aut(W)$. This gives rise to the natural  group homomorphism $\Aut(W)\to \Aut(W^{\mathrm{ns}})$, which is injective, since $W^{\mathrm{ns}}$ is dense in $W$ in Zariski topology. This implies that in the course of the proof of  \thmref{main}  and  \corref{main3d}   we may assume that $W$ is {\sl smooth}.
\end{remark}

The paper is organized as follows. Section \ref{s1} contains notation  and  auxiliary results about  fiberwise automorphisms of fibered varieties. In Section \ref{s2} we discuss automorphism groups of varieties that are birational to a product $A \times \BP^1$ where $A$ is a smooth rigid projective variety. Section \ref{s3} contains the proof of  \thmref{main} and  \corref{main3d}

\subsection{Acknowledgements} We are grateful to Shulim  Kaliman,
 Michel Brion, and Vladimir  Berkovich  for helpful discussions. Our thanks go to the referees,  whose comments helped to improve the exposition.

\section{Preliminaries}
\label{s1}

If $X$ is an irreducible  algebraic variety over $k$ then 
\begin{itemize}

\item
We write $k[X]$ for the ring of regular functions on $X$ and
$k(X)$ for its field of rational functions. In this case one may view $\Bir(X)$ as the group of all $k$-linear automorphisms  of $k(X)$ and $\Aut(X)$ as a certain subgroup of $\Bir(X)$. We write $\mathrm{id}_X$ for the identity automorphism of $X$, which may be viewed as the identity element of groups $\Aut(X)$ and $\Bir(X)$.
\item

 By points of $X$ (unless otherwise stated)
we always mean $k-$points.  A {\sl  general point} means  a point of an open dense 
subset of $X.$  

\item
If $X$ is smooth then 
$K_X$ and   $q(X)$ stand for the canonical class of $X $ and 
irregularity $h^{0,1}(X)$   of $X$, respectively. 


\item  $ \BC,$  $\BQ$ and $\BZ$ stand for fields of complex numbers,   the rationals, and ring of integers, respectively.

\item If $F$ is a field  then we write $\ov F$ for  its algebraic closure.


\item
Let $X,Y ,T $  be irreducible varieties,    $p:X\to T, \ q:Y\to T$   morphisms.  
We say that a rational map  $f:X\dasharrow Y$ is  $p,q- $fiberwise  if there exists morphism $g_f:T\to T$ such that
 the following diagram commutes:
\begin{equation}\label{diagram000}
\begin{CD}
X  @>{f}>>  Y \\
@Vp VV @Vq VV \\
T @>{g_f}>>T
\end{CD}.
\end{equation}
\item
If $X=Y, \ p=q , \  f\in \Bir (X), $ then we say that $f$ is $p-$fiberwise   and denote by $\Bir_p(X)$ the group of all $p-$ fiberwise birational automorphisms of $X.$   
We write $\Aut_p(X)$ for the intersection of     $\Bir_p(X)$ and $\Aut(X)$ in $\Bir(X)$, which is the  group of all $p-$ fiberwise automorphisms of $X.$
\item   Recall that if a smooth projective variety $A$ is rigid,   then   any rational map from a smooth variety to $A$ is a morphism  (\cite[Corollary 1.44]{Deba}).  In particular,
 $\Bir(A)=\Aut(A).$ Abelian varieties and bielliptic surfaces are rigid.
\end {itemize}

We start with an  auxiliary

\begin{Lemma}\label{lem1}  Assume that $U,V$ are smooth 
irreducible
quasiprojective varieties endowed by a  surjective morphism 
$p:U\to V$ such that the fiber $P_v:=p^{-1}(v)$ is projective and irreducible 
for every point $v\in V.$
Assume that $C\subset U$ is a closed subset and that $C\cap P_v$ is a finite set for every 
point  $v\in V.$
 Assume that $f\in \Aut _p (U\setminus C).$

Then $f\in \Aut _p(U).$
\end{Lemma}

\begin{remark}  In loose language this Lemma asserts that every   fiberwise  automorphism $f\in \Aut (U\setminus C)$ may be extended to  an automorphism of $U$ if $C$ has only  ``$p-$horizontal "  components  over $V.$ \end{remark}

\begin{proof}

 Take any smooth projective closure $\ov V$  of $V$ and choose such a smooth projective closure $\ov U$ of $U$
that the rational extension $\ov p: \ov U\to \ov V$  of $p$ is a morphism. Since all the fibers of $p$ are projective and irreducible, 
 we have 
$\ov  p^{-1}(V)=U,  \ \ov  p^{-1}(\ov V\setminus V)=\ov U\setminus U$  and $\ov p  \bigm  |_U=p$
(see, for example, \cite[Section 2.6] {MuOd}). 
Let  $\ov f: \ov U\dasharrow \ov U$ be the rational extension of $f.$ 
Let $(\tilde U', \tilde f,\pi)$ be a resolution of indeterminacy of $\ov f.$ Let  $\ov {g}_f\in \Bir(\ov V)$ be an  extension of $g_f.$

We have  a commutative diagram 
\begin{equation}\label{diagram00}
\begin{aligned}
& &   & \tilde U '       &  {}  & {}        &   & & \notag \\
& &\pi& \downarrow  &        {\searrow}                     &  \tilde f  & & &\notag \\
& & &\ov U        &\stackrel{\ov f}{\dashrightarrow } &\ov U   & & &\notag \\
&  & \ov p& \downarrow  &                              &  \downarrow &\ov p& &\notag \\
& &  &\ov V            &\stackrel {\ov g_f}{\dashrightarrow} &\ov V  &&    &        \end{aligned},  \end{equation}

Since $g_f$ is an automorphism of $V$, we 
have 
$$\tilde  U:=(\ov p\circ\pi)^{-1}(V)=(\ov p\circ\tilde f)^{-1}(V)=\pi^{-1}(U)$$
 and 
 we may restrict 
the maps  $\pi,  \ \tilde f,  \ \ov f, \ \ov p, \ \ov {g}_f$
to quasiprojective varieties  $\tilde U, U,$ and $V$ and obtain the following commutative diagram:
\begin{equation}\label{diagram}
\begin{aligned}
& &   & \tilde  U        &  {}  & {}        &   & & \notag \\
& &\pi& \downarrow  &        {\searrow}                     &  \tilde f|_{\tilde U}  & & &\notag \\
& & &U        &\stackrel{\ov f|_U }{\dashrightarrow } &U   & & &\notag \\
&  & p& \downarrow  &                              &  \downarrow & p& &\notag \\
& &  &V            &\stackrel {g_f}{\longrightarrow} &V  &&    &        \end{aligned},  \end{equation}


Here
  \begin{itemize}\item  $\pi$ and $\tilde f:=\tilde f|_{\tilde U}$ are morphisms; 
\item   $f:=\ov f|_U \in \Aut(U\setminus C)\cap \Bir(U);$  
\item   $\pi$  is an isomorphism of $U_1:=\pi^{-1}(U\setminus C)$ to $U\setminus C.$ 
\end{itemize}

We have to show that $f $ is defined at all points of $C$.  For this, we need to check that $\tilde
f(\pi^{-1}(c))$ is a point for every point $c\in C.$
Since $\pi$ and $\tilde f$ are birational morphisms, the sets $\tilde f^{-1}(a)$  and $ \pi^{-1}(a)$
are  connected for every point $a\in U$ by the Zariski Main Theorem  (see \cite{MumfordRed},  Chapter III, \S 9).  Take   $a\in U\setminus C.$  Then $\tilde f^{-1}(a)$ contains an isolated point  $\pi^{-1}(f^{-1}(a))\in U_1,$  which 
(by the Zariski Main Theorem)
is the only  connected component of  $\tilde f^{-1}(a).$  Thus 
$\tilde f^{-1}(U\setminus C)=U_1, \ \tilde f^{-1}(C)=\pi^{-1}(C), $  or
$ \tilde f(\pi^{-1}(C))=C, \ \tilde f (U_1)=U\setminus C.$  Hence for every point $c\in C$  we have 
$$\tilde f(\pi^{-1}(c))\subset C\cap P_{g_f(p(v))}$$ and the latter is a finite set.  Since $\tilde f(\pi^{-1}(c))$ has to be irreducible, it is 
a single point.  Thus $f=\tilde f\circ \pi^{-1}$ is defined at every point of $U.$
\end{proof}

\begin{Lemma}\label{sides} Assume that a group $G$ sits in the short exact sequence
$$\{0\}\to  G_1 \to G\to G_2\to \{0\}.$$ 
Suppose that one of the following two condition holds.

\begin{itemize}
\item[(1)]
 $G_1$ is bounded and  $G_2$ is strongly Jordan.
\item[(2)]
$G_1$ is strongly Jordan  and $G_2$ is bounded.
\end{itemize}

Then $G$ is strongly Jordan.\end{Lemma}

\begin{proof}
Suppose (1) holds.  Then a lemma of Anton Klyachko \cite[Lemma 2.1]{BZ16} implies that $G$ is strongly Jordan.

Suppose (2) holds. Then both $G_1$ and $G_2$ are quasi-bounded. By Remark \ref{exactR}, $G$ is also quasi-bounded.
It follows from \cite[Lemma 2.3(1)]{MengZhang} that $G$ is Jordan. This implies that $G$ is strongly Jordan.

\end{proof}

\begin{remark}
\label{MZstrong}
Let $X$ be a projective variety. It is actually proven in \cite{MengZhang} that $\Aut(X)$ is strongly Jordan (not just Jordan):
the assertion follows readily from the combination of \cite[Theorem 1.4]{MengZhang}and \cite[Lemma 2.5]{MengZhang}.
\end{remark}

In the next Proposition   we   consider  the group $\Aut_p(X)$ where $p:W\to A $ is a morphism from a smooth 
quasiprojective variety $W$  with projective fibers
and $A$ is a  smooth  rigid  projective variety.

 \begin{Proposition}\label{m0}  Suppose that    $A$ is a smooth  rigid  projective variety of positive dimension.  
Let X be  a smooth 
irreducible
projective variety and $p:X\to A$ morphism such that the generic fiber
  (and, hence, the fiber over  a general point $a\in A$)  is connected. 
 Let $S\subsetneq A$ be a  closed subset of $A.$ Put $Z=p^{-1}(S)$ and $W=X\setminus Z.$  
Then the group $H=\Aut_p(W)$ is 
strongly 
Jordan.
\end{Proposition}
\begin{remark}\label{irreducible} Let $A_r$ be the set of all points $a\in A\setminus S$ such that   the fiber $p^{-1}(a)$ is  smooth  (hence, 
 irreducible).  Then $W_r:=p^{-1}(A_r)$ is evidently $H-$invariant and  $H$ is embedded in $\Aut(W_r).$  Thus while proving the Proposition we  may assume that for every point   $a\not\in S$  the fiber $p^{-1}(a)$ is irreducible.\end{remark}

\begin{proof}
If $S=\emptyset$ then $W=X$ is projective and the desired result follows from  
results of  \cite   {MengZhang} (and Remark \ref{MZstrong}). 
Thus we assume that $S\ne\emptyset.$  Then
 \begin{itemize}\item We denote by
  $G(A):=\Aut(A)$  be  the group of automorphisms  of $A.$  
 \item   We denote by  $G(S)\subset G(A)$  the subgroup of all elements $g\in G(A)$  such that  $g(S)=S;$
\item The identity
  component $G(A)_0$  of $G(A)$  
   is a connected algebraic group (\cite[Corollary 2]{Mat}) ;   
\item 
The intersection
$G_S=G(S) \cap G(A)_0$ is a closed subgroup of  $G(A)_0$, because $S$ is a closed subset of $A;$
\item   The identity component  $G_0$ of $G_S$ is a closed subgroup  in $G_S,$  thus it is a  connected  algebraic group, and
 has  finite index in $G_S;$
\item  The  factor group $G(S)/(G_S)$ is bounded ( \cite[Lemma  2.5]{MengZhang});
\item Hence,   the group $G(S)/G_0$ is bounded;
\item Since $G_0$  acts on a non-uniruled projective variety $A, $  it contains no non-trivial  connected linear algebraic
subgroup  (otherwise, the open dense subset of $A$  would be covered by rational orbits).
Thus it is isomorphic to an abelian variety  by the Chevalley's Theorem (\cite {C}).
\end{itemize}

By definition, for  every automorphism $f\in \Aut_p(W)$ there is  $g_f\in \Aut(A\setminus S)$  that
 may be included into the following commutative diagram : 

\begin{equation}\label{diagram2}
\begin{CD}
 W @>{f}>>  W  \\
@Vp VV @Vp VV \\
A\setminus S   @>{g_f}>>A\setminus S   \\
\end{CD}
\end{equation}

Hence,   the group  $H=\Aut_p(W)$ sits in the following exact sequence

 \begin{equation}\label{seq0}
0\to H_i\to   H \to  H_a\to 0,\end{equation}
 where 

--$ H_i=\{f\in H \ | \ g_f={\mathrm{id}_A}\}$  is a subgroup of  the  automorphism group of the generic fiber  $\mathcal  W_p$   of $p;$

--$ H_a=\{g\in \Aut(A\setminus S)    \   | \  g=g_f  \    \text {for some} f\in H\}.$

Note that 
we have 
  \begin{itemize} 
\item $H_a\subset G(S),$    since   $\Bir(A)=\Aut(A);$ 
\item  Every $g\in H_a\subset G(S)$ moves a $G_0$-orbit (in $A$) to a $G_0$-orbit,  since $G_0$   is  a closed normal subgroup of $G(S);$
\item  The orbit $G_0(z)$ of a point $z\not\in S$ is a projective subset of $A,$  since $G_0$ is an abelian variety;
 \item The orbit $G_0(z)$ of a point $z\not\in S$  does not meet $S.$ Hence, if $z\not\in S$ then
$p^{-1}(G_0(z))\cap Z=p^{-1}(G_0(z)\cap S)=\emptyset,$  i.e  $p^{-1}(G_0(z))$  is a closed irreducible projective subset of $W.$
 Indeed
 it is a fibration with irreducible projective fibers over  a projective orbit   $G_0 (z)$  ( \cite[Chapter 1, n.6.3, Theorem 8] {Shaf}). 
\end{itemize}

By a theorem of M. Rosenlicht \cite {Ros},  there exist a dense  open $G_0-$invariant subset   $U\subset  A, $ 
a  quasiprojective  variety $V$ and a morphism $\pi:U\to V$ such that a fiber $\pi^{-1}(v)$ is precisely an orbit of $G_0$  for every $v\in V.$
That means that  $V$ is a geometric quotient of $U$ by the $G_0$-action.  Since $S$ is $G_0-$invariant, we may assume that $U\subset A\setminus S.$
Since every $g\in H_a\subset G(S)$ moves a $G_0$-orbit (in $A$) to a $G_0$-orbit, the  map $h_g:=\pi\circ g\circ\pi^{-1}: V\to V$
is defined at every points of $V,$  hence is a morphism (see \cite[Lemma 10.7 on pp. 314--315]{Itaka}) .
Moreover, 
the following  diagram  commutes.

\begin{equation}\label{diagram22}
\begin{CD}
 W @>{f}>>    W  \\
@Vp VV @Vp VV \\
U   @>{g_f}>>U   \\
@V\pi VV @V\pi VV \\
V @>{h_{g_f}}>>V
\end{CD}
\end{equation}

  Let $\tau=\pi\circ p.$   We have $f\in \Aut_{\tau} (W).$
Since the general  fiber  $T_v=\tau^{-1}(v), v\in  V $ is a  projective irreducible variety,   the generic fiber 
$ \mathcal  {T} $ of $\tau$ is projective and irreducible as well  (\cite [Proposition 9.7.8.]{EGA43}).


Moreover, we have the following exact sequence of groups
 \begin{equation}\label{seq2}
0\to  H_T\to H\to H_V\to 0,\end{equation}
 where 

--$H_T=\{f\in H \ | \ h_{g_f}={\mathrm{id}_V}\};$
--$H_V=\{h\in \Aut(V) \ \   | \  h=h_{g_f}  \    \text {for some} \  f\in H\}$.
 (In particular case of $G_0=\{\mathrm{id}_A\}$  we have   $V=A\setminus S, $  $\pi=\mathrm{id}_A,$  $h_{g_f}=g_f ,$ and $H_V=H_a\subset G_S.$) 
The group $H_T\subset \Aut(\mathcal  {T})$ is strongly  Jordan, according to \cite
[Theorem 1.4, Lemma 2.5]{MengZhang} 
(and Remark \ref{MZstrong}). 
 The group $H_V$ is isomorphic to a subgroup of $ G(S)/G_0,$ hence  is bounded.  
Therefore, by \lemref{sides},       $H$ is  strongly  Jordan.

\end{proof}

\section{Admissible triples and related exact sequences}
\label{s2}
Let $n$ be a positive integer and $A$ be a $n$-dimensional irreducible smooth rigid projective variety
(e.g,  an abelian  variety or a product of curves of positive genus).  We write $\mathcal{K}$ for $k(A)$.

Let us define an $A$-{\sl admissible triple} as a triple $(X, \phi,Z )$  that consists of a smooth irreducible  projective variety $X$,  a birational isomorphism $\phi :X\dasharrow   A\times \BP^1$ and a  closed subset $Z \subsetneq X$.  We denote by $W$ the  open subset
$$W=X\setminus Z\subset X.$$

 We will freely use the following  notation and properties of admissible $A$-triples.

\bigskip

\begin{description}

\item[a]  Let  $p_A:A\times \BP^1\to A$  be the projection map on the first factor. Then the composition
$p:=p_A\circ \phi:X\to A$ 
is a morphism,
since $A$ is rigid.
We say that $p$ is induced by $\phi.$

\item[b] Since 
$X$ is birational to $A\times\BP^1$,  there is an open non-empty subset $B\subset A$  such that $\phi$ induces an isomorphism between $X_B=p^{-1}(B)$ and $B\times\BP^1.$ 
(This follows from the fact that indeterminancy locus of $\phi$ has codimension $\le 2$ in $X, $  thus it is mapped by $p$ into a proper closed subset  of $A.$)

Moreover,  $\phi$ is $p,p_A-$fiberwise:  the following diagram  commutes.
 \begin{equation}\label{diagram0}
\begin{CD}
X_B  @>{\phi}>>  B\times\BP^1   \\
@Vp VV @Vp_A VV \\
B @>{id}>> B
\end{CD};
\end{equation}
\bigskip
  \item[c]  It follows from \eqref{diagram0} that the general fiber $P_x:=p^{-1}(x)$
 (i.e. fiber over a point  $x$ of a certain open dense subset of $A$)
 is isomorphic to $\BP^1;$

\item[d] Let us put:\begin{itemize}\item
$r(Z)$ - the number of irreducible over $\mathcal{K}$ components of $Z$ that are mapped dominantly onto $A;$  we will call such components "horizontal";
\item
$m(Z)$- the degree of the restriction of $p$ to $Z,$  i.e the number of 
 $\overline {\mathcal K}$-points in $p^{-1}(a)$
for a general point $a\in A.$\end{itemize}

\item[e]
The generic fiber $\mathcal X_p$  of $p$  is isomorphic to  the projective line $\BP^1_{\mathcal{K}}$ over $\mathcal{K}$; the generic fiber  $\mathcal W_p$  of the  restriction $p\bigm |_W \to A$ (of $p$ to $W)$   is isomorphic to $\BP^1_{\mathcal{K}}\setminus M, $  where $M$ is a finite set that is defined over $\mathcal{K}$ and consists of
$m(Z)$
points that are  defined over a finite algebraic extension of  $\mathcal{K}$. In other words, 
 the $\mathcal{K}$-variety $\mathcal W_p$ is isomorphic to the projective line over $\mathcal{K}$ with $m(Z)$ punctures. In particular, the group $\Aut_{\mathcal{K}}(\mathcal W_p)$ is finite if $m(Z) >2$. On the other hand, $r(Z)$ is the number of Galois orbits in the set of punctures. In particular,
$$1 \le r(Z)\le m(Z) \ or \ 0=  r(Z)=m(Z).$$

\item[f]
We may choose  $B$ in such a way that $Z_B:=Z\cap X_B$ 
meets every fiber $P_b,  \ b\in B,$ at precisely  $m(Z)$    $\overline {\mathcal K}-$ points.
  In particular, $Z_B$ is a finite cover of $B$.

\item[g] 
 Every birational map  $f\in \Bir(X)$ is $p-$fiberwise : we denote by  $g_f\in \Bir(A)$ the corresponding  automorphism $g_f:A\dasharrow A$ (see  \cite{BZ16}).
Since $A$ is rigid, $g_f$ actually belongs to $\Aut(A)$.  This implies that $$\Aut(W)=\Aut_p(W).$$
(Here $p$ denotes the restriction of $p: X \to A$ to $W\subset X$.)

\item [h]  Let us consider the subgroups
$$H_i=\{f\in \Aut(W) \ |   \ g_f=\mathrm{id}_A\}\subset \Aut(W)$$ 
 and 
 $$H_a=\{g\in \Aut(A)  \ |  \ g=g_f\text{  for some} \  f\in \Aut(W)\} \subset \Aut(A).$$ 
We have the following short exact sequence of groups.
\begin{equation}\label{seq1}
0\longrightarrow H_i\longrightarrow \Aut(W)\longrightarrow H_a\longrightarrow 0\end{equation}

\item[j]  Group  $H_i$ is isomorphic to a subgroup of $\Aut_{\mathcal{K}}(\mathcal   W_p).$   Thus it is Jordan;   it   is finite if $m(Z)>2$.

\end{description}

\begin{remark}
\label{redtriple}
Let $W$ be a smooth  quasiprojective irreducible variety that is birational to $A\times \BP^1$.  Then there is an $A$-admissible triple $(X, \phi,Z )$ such that $X\setminus Z$ is biregular to $W$. Indeed, one may take as $X$ any smooth projective closure of $W$ and put $Z=X\setminus W$.

\end{remark}

\begin{Lemma}  Suppose that $A$ is an irreducible smooth projective  variety that is not uniruled. (E.g., $A$ is rigid).
Then $\Bir(A\times\BP^1)$ is quasi-bounded.
\end{Lemma}
\begin{proof}
Let $p_A: A\times\BP^1\to A$ be the projection map. Its generic fiber $\mathcal{X}$ is the projective line $\BP^1_{k(A)}$ over $k(A)$.
Each $u \in \Bir(A\times\BP^1)$ is a $p_A$-fiberwise, see \cite[Lemma 3.4 and Cor. 3.6]{BZ16}.  By \cite[Cor. 3.6]{BZ16}, 
$\Bir(A\times\BP^1)$ sits in an exact sequence
$$\{0\} \to \Bir_{k(A)}(\mathcal{X}) \to \Bir(A\times\BP^1) \to \Bir(A).$$
Actually, $\Bir(A\times\BP^1) \to \Bir(A)$ is surjective, because one may lift any birational automorphism of $A$ to a birational automorphism of $A\times\BP^1$.
On the other hand, since $\mathcal{X}$ is the projective line, $\Bir_{k(A)}(\mathcal{X})$ is the projective linear group $\PGL(2,k(A))$. This gives us a short exact sequence
\begin{equation}
\label{ARexact}
\{0\} \to \PGL(2,k(A)) \to \Bir(A\times\BP^1) \to \Bir(A) \to \{0\}.
\end{equation}
The theorem of Jordan implies that the linear group $\PGL(2,k(A))$ is strongly Jordan. In particular, it is quasi-bounded. On the other hand, since $A$ is {\sl not} uniruled,
$\Bir(A)$ is also quasi-bounded (\cite[Remark 6.9]{PS1}, \cite[ Proof of Cor. 3.8 on p. 236]{BZ16}.  It follows from \eqref{ARexact} and Remark \ref{exactR} that $\Bir(A\times\BP^1)$ is also quasi-bounded.
\end{proof}

\begin{Lemma}\label {field} Assume that  $m(Z)=2$   and $r(Z)=1. $ Let $\mathrm{Tor}(H_i)$
be the set of all elements of $H_i$ of finite order.
 Then the following conditions hold.
 \begin{itemize}
 \item[(i)]
 $\mathrm{Tor}(H_i)$ consists of, at most, 4 elements.
 \item[(ii)]
  every  element of finite order in  $H_i$ has order 1 or 2. 
  \item[(iii)]
  Every finite subgroup of $H_i$ is abelian and its order divides 4 while its exponent divides 2.
  \end{itemize}
\end{Lemma} 
\begin{proof}   Let $(u_0:u_1)$ be homogeneous coordinates in $\BP^1_{\mathcal{K}}.$ Since $m(Z)=2$, and $Z$ has only one irreducible component $Z_1$  over $\mathcal{K}, $ we may assume that
$Z_1$  is  defined by equation $(u_0-\mu_1u_1)(u_0-\mu_2u_1)=0,$
where 
$\mu_1, \mu_2$  are  distinct elements   of   a quadratic extension $\mathcal{K}_2$ of $\mathcal{K}$ that are conjugate over $\mathcal{K}$.

 Every automorphism $f \in \mathrm{Tor}(H_i)$  of $\mathcal W_p$ may be extended uniquely to a  periodic  automorphism  $\ov f$
of $\mathcal X_p\cong \BP^1_{\mathcal{K}}.$ The 2-element subset 
$$\{\mu_1, \mu_2\}\subset \BA^1(\mathcal{K}_2)\subset \BP^1(\mathcal{K}_2)$$
is $\ov f$-invariant for all $f \in H_i$. This means that either $\ov f$ leaves invariant both $\mu_i$ or permutes them. 

Put   $z=\frac{(u_0-\mu_1u_1)}{(u_0-\mu_2u_1)}\in \mathcal{K}_2(A).$   The extension $\ov f$
leaves the set $\{z=0\}\cup\{z=\infty\}$ invariant for all $f \in H_i$. Thus, $\ov f(z)=\lambda z$ or 
 $\ov f(z)=\frac{\lambda}{ z}$   for suitable nonzero $\lambda \in \mathcal{K}_2.$  In both cases 
$(\ov f)^2(z):=(\ov f\circ\ov f)(z)=\lambda^2z,$
This implies that $\lambda$ is a {\sl root of unity} if $\ov f$ is {\sl periodic}, i.e., if $f\in \mathrm{Tor}(H_i)$; in particular $\lambda \in \mathcal{K}.$ 
Suppose that 
    $\ov f(u_0:u_1)=(u_0':u_1').$
Then one may easily check that    either (in the former case)
\begin{equation}\label{field1}
\frac{(u_0'-\mu_1u_1')}{(u_0'-\mu_2u_1')}=\lambda\frac{(u_0-\mu_1u_1)}{(u_0-\mu_2u_1)}
\end{equation}
or (in the latter case)
\begin{equation}\label{field2}
\frac{(u_0'-\mu_1u_1')}{(u_0'-\mu_2u_1')}=\lambda\frac{(u_0-\mu_2u_1)}{(u_0-\mu_1u_1)}.
\end{equation}

In order for  these maps to be defined over $\mathcal{K}$  the matrices  (respectively) 
$$\begin{pmatrix}\mu_1-\lambda\mu_2 &\mu_1\mu_2(\lambda-1)\\
(1-\lambda) &\lambda\mu_1-\mu_2 \end{pmatrix} \ \text{ and} 
\begin{pmatrix}\mu_1-\lambda\mu_2 &\lambda\mu_2^2-\mu_1^2\\
(1-\lambda) &\lambda\mu_2-\mu_1 \end{pmatrix}$$
should be defined 
 (up to multiplication by a nonzero element of $\mathcal{K}_2$) 
over $\mathcal{K}$  as well.  Since $\lambda\in \mathcal{K}, $ it may happen only if $\lambda=\pm 1.$  
This implies that $(\ov f )^2: z \mapsto \lambda^2 z$ is the identity map, i.e.,  
 the order of $\ov f$ is  either $1$ or $2$. In addition, there are, at most, four   elements   in  $\mathrm{Tor}(H_i).$  Namely, (written in $z-$ coordanate)
$$\ov f(z)=z, \ \ov f(z)=-z, \ov f(z)=\frac{1}{z}, \ \ov f(z)=-\frac{1}{z}.$$
\end{proof}


%
%
%
%
One  may see  \lemref{field}   in a more general way.
Let $\mathcal{K}$ be a field of characteristic zero that contains all roots of unity. Let $n \ge 2$ be an integer.

The following assertion is an easy application of Kummer theory  \cite[Chapter VI, Section 8]{Lang}.

\begin{theorem}
\label{kummer}
Let $u$ be a matrix in $\GL(n,\mathcal{K})$, whose image $\bar{u}$ in $\PGL(n,\mathcal{K})$ has finite order. Suppose that $u$ has an eigenvalue that does not belong to $\mathcal{K}$. Then there is a positive integer $d$ such that $d\mid n$ and all eigenvalues  of $u^d$ lie in $\mathcal{K}$.  In addition, if $n$ is a prime then $\bar{u}$ has order $n$.
\end{theorem}

\begin{proof}
We know that there are a positive integer $m$ and a nonzero element $a \in \mathcal{K}$ such that $u^m=a \cdot \mathbf{1}_n$
where $\mathbf{1}_n$ is the identity square matrix of size $n$.
Clearly, the order of $\bar{u}$ is strictly greater than $1$ and divides $n$.

 Let $\alpha$ be an eigenvalue of $u$ that does not belong to $\mathcal{K}$. Then $\alpha^m=a$. Let us consider the finite algebraic  field extension  $\mathcal{K}^{\prime}=\mathcal{K}(\alpha)$  of $\mathcal{K}$ and denote by $d$ its degree $[\mathcal{K}^{\prime}:\mathcal{K}]$. Clearly, $d>1$. The Kummer theory tells us that $\mathcal{K}^{\prime}/\mathcal{K}$ is a a cyclic extension and $d \mid m$. In other words, $\mathcal{K}^{\prime}/\mathcal{K}$ is Galois and its Galois group $G$ is cyclic of order $d$.  If $\beta$ is another eigenvalue of $u$ then
$$\beta^m=a=\alpha^m$$
and therefore the ratio $\beta/\alpha$ is an $m$th root of unity and therefore lies in $\mathcal{K}$. This implies that 
$$\mathcal{K}(\beta)=\mathcal{K}(\alpha)=\mathcal{K}^{\prime};$$
in particular, none of eigenvalues of $u$ lies in $\mathcal{K}$.

Recall that the  cardinality of $G$  coincides with $d$. Since $\beta$ generates  $\mathcal{K}^{\prime}$ over $\mathcal{K}$, the set $\{\sigma(\beta)\mid \sigma \in G\}$ consists of $d$ distinct elements, each of which is an eigenvalue of $u$ and has the same multiplicity. Since the spectrum of $u$ is a disjoint union of  $G$-orbits, $d$ divides $n$.

Take   an element $\tau$   of (abelian  group) $G$.  Then $\tau(\beta)=\zeta \beta$ where $\zeta$ is a root of unity that lies in $\mathcal{K}$. The norms of conjugate  $\beta$ and $\tau(\beta)$ (with repect to $\mathcal{K}^{\prime}/\mathcal{K}$ ) do coincide. This means that
$$\prod_{\sigma \in G}\sigma(\beta)=\prod_{\sigma \in G}\sigma(\tau\beta)=\zeta^d \prod_{\sigma \in G}\sigma(\beta).$$
It follows that 
$$\zeta^d=1, \  \tau(\beta^d)=(\tau\beta)^d=\beta^d$$
for all $\tau\in G$. This implies that $\beta^d\in \mathcal{K}$ for all eigenvalues $\beta$ of $u$ and therefore all eigenvalues of $u^d$ lie in $\mathcal{K}$.

Now assume that $n$ is a prime. Then $d=n$ and counting arguments imply that the spectrum of $u$ consists of exactly one $G$-orbit say, $G\beta$. Then all the eigenvalues of $u^n=u^d$ coincide with 
$$\beta^n=\beta^d\in \mathcal{K}.$$
This implies that $u^n$ is a scalar and therefore the order of $\bar{u}$ divides $n$. One has only to recall that this order is greater than $1$ and $n$ is a prime.
\end{proof}

The next lemmas show that the case $m(Z)\le 2$ may be reduced to the case $m(Z)=0$.

To this end we  find  an open   $H_a-$invariant  subset  $\tilde B\subset A$  of  $A$   such that 
$\tilde W=p^{-1} (\tilde B)$ is a complement of exactly two   (respectively 1) "horizontal"  components . Namely, 
we build a rank two vector bundle $E$ over  $\tilde B$ such that 
$\tilde W$ appears to be  isomorphic  to the complement $Y\setminus D$ 
 of   two ( respectively, one) disjoint sections  in $Y:=P(E).$

More precisely, we are going to build the following chain of   maps and  inclusions  of smooth irreducible quasiprojective varieties
\begin{equation}\label{diagram-m<2}
\begin{aligned}
& X&\hookleftarrow \quad   & W&\hookleftarrow \quad     &    \tilde  W& \stackrel{\Psi}{\rightarrow }\quad & (Y\setminus{D})  &\hookrightarrow\quad  &Y&\hookrightarrow\quad& \tilde X&\notag \\
&\downarrow p&{}&\downarrow p&{}&\downarrow p&{}                                       &\quad \downarrow \pi&{} &\downarrow  \pi&{} &
\downarrow\tilde \pi&\notag \\
&A&=\quad    &A&\hookleftarrow\quad     &    \tilde B& =   \quad                                      & \quad    \tilde B&  = \quad   &   \tilde B& \hookrightarrow\quad  &A       &
 \end{aligned},  \end{equation}

such that:
 \begin{itemize}\item ${\tilde B}$ is an open dense subset of $A$   invariant under the $H_a$-action;
\item $\tilde  W=p^{-1}(\tilde B)\subset W$ is invariant under the action of $\Aut(W)$;
\item  $ \Psi$ is an isomorphism;
\item every fiber of $\pi: Y\to\tilde B$ is projective;
\item $D$ is a closed subset of $Y$ that meets every fiber of $ \pi$ at no more than two points;
\item $\tilde X$ is projective;
\item $\tilde  \pi(\tilde X\setminus Y)=A\setminus \tilde  B.$
\end{itemize}

According to \lemref{lem1},  $\Aut(Y\setminus{D})\subset \Aut(Y).$   
Thus, instead of  $\Aut(W)$ we may study  $\Aut(Y)$ where $Y$ is fibered over ${\tilde B}\subset A$ with projective fibers
 (hence  $m(\tilde X\setminus Y)=0$).

The building of this construction is  done in the following Lemmas. 

\begin{Lemma}\label{m1} If      $(X, \phi,Z )$
is an $A$-admissible triple, $W:=X\setminus Z$
  and 
$m(Z)=1$ then  there exists an $A$-admissible triple  $(\tilde X, \tilde \phi , \tilde Z)$
with $m(\tilde Z)=0$ and a group embedding $\Aut(W)\hookrightarrow \Aut (\tilde W),$ where  $\tilde W:=\tilde X\setminus \tilde Z.$\end{Lemma}

\begin{proof}

 Let $(w_0: w_1)$ be homogeneous coordinates in $\BP^1.$  We may choose them in such a way that  $\phi (Z_B)=B\times\{(0:1)\}.$

Let
\begin{itemize}\item
 $W_B=W\cap  X_B=X_B\setminus Z_B;$ 
\item  
$B_g=g(B)$   for an automorphism $g\in H_a;$ 
\item  $ W_g=W\cap  p^{-1}(B_g);$
\item $\tilde B=\cup B_g, \ g\in H_a;$
\item  $\tilde W= p^{-1}(\tilde B)=\cup W_g, \ g\in H_a.$
\end{itemize}
We have $\phi(W_B)=B\times (\BP^1\setminus\{w_0=0\}). $
Thus, the rational function 
$t=w_1/w_0$
  is defined on $\phi(W_B)$ and the rational function 
$\tau=\phi^*(t)$ is defined on $W_B.$  It establishes   an isomorphism of the fiber 
$p^{-1}(b)\cap W$ with $\mathbb{A}^1_t$   if $b$ is a point of $B$.

For every ${g\in H_a}$ there exists $f_g\in \Aut(W)$ such that $g=g_{f_g}$  and $f_g(W_B)=W_g.$
We define $\tau_g  = \tau\circ  f_g^{-1}$, which is a regular function on $W_g$.
  (Note that  apriori the choice  of $f_g$  is not unique.  A different choice  
   of $f_g$ will  change $\tau_g$ by a $p-$fiberwise automorphism  of  $W_B$, which becomes
   a nondegenerate affine  transformation of the generic fiber  $\mathcal W_p\sim\mathbb{A}^1_{\mathcal{K}}$.)

We introduce the isomorphisms $\psi_g:W_g\to B_g\times \mathbb{A}^1$ by $\psi_g(w)=(p(w),\tau_g(w)).
$  Actually, 
$\psi_g$  are compositions of the chain of automorphisms
$$W_g              \stackrel{f_g^{-1}}{\rightarrow }   W_B                \stackrel{\phi}{\rightarrow } \quad B\times \mathbb{A}^1 _t  
\stackrel{(g^{-1},id)}{\rightarrow } B_g\times \mathbb{A}^1_{t}.$$

Note that in this chain $f_g^{-1}$ is $p-$fiberwise,   $\phi$  is $p,p_A-$fiberwise,  and $(g^{-1},\mathrm{id})$ is $p_A-$fiberwise,  thus 
$\psi_g$  is $p,p_A-$fiberwise.   It  may be included into the following commutative diagram

\begin{equation}\label{psig}
\begin{CD}
W_g   @>>> B_g\times \mathbb{A}^1_{t} \\
@Vp VV @Vp_A VV \\
B_g @>{\mathrm{id}}>>  B_g
\end{CD}.
\end{equation}

  If $g,h\in H_a$  and $w\in  W_g\cap W_h$ then:
\begin{itemize}\item $b=p(w) \in(B_g\cap  B_h);$ 
\item    functions $\tau_g$ and $\tau _h$ provide an isomorphism  of the  fiber $P_b=p^{-1}(b)\cap W$ 
with $\mathbb{A}^1_{t} $  hence 
$$\tau_g=\tau\circ  f_g^{-1}= \tau\circ f_h^{-1}\circ  f_h\circ f_g^{-1}=\tau_h\circ  f_h\circ f_g^{-1}=\tau_h\alpha+\beta, $$
  where $\alpha:=\alpha_{gh}(b), \beta:=\beta_{gh} (b)$ are regular in $B_g\cap  B_h,$ constant along $P_b,$
 and $\alpha_{gh}$  does {\sl not} vanish in $(B_g\cap  B_h)$;
\item $\Psi_{gh}=\psi_g(w)\circ\psi_h^{-1}$ is a $p_A-$fiberwise automorphism of $(B_g\cap B_h)\times \mathbb{A}^1$
defined by $\Psi_{gh}(b,\tau_h)=(b,\tau_g)=(b, \alpha _{gh}(b)\tau_h+\beta_{gh}(b));$
\end{itemize}

  It follows that $\tilde W$ is the total body of an $\mathbb{A}^1$-bundle on $\tilde B$: the latter is  defined by
 transition functions  $\Psi_{gh}$.

We define a rank two vector bundle by the following data.

\begin{itemize}\item the covering of  $\tilde B $ by the open subsets $B_g, g \in H_a;$
\item natural projection $ \pi_E:B_g\times \BA^2_{(u_0,u_1)}\to B_g;$
\item  transition matrices on 
$B_g\cap  B_h$
$$M_{gh}=\begin{pmatrix}1&0\\\beta_{gh}&\alpha_{gh}\end{pmatrix}.$$
\end{itemize}
The    maps  
$$\{\ov \psi_g(w):  W_g\to \mathbb{P}(E),  \ \ov \psi_g(w)=(p(w), (1:\tau_g(w))$$
 glue together to an isomorphism
  $$\Psi:{\tilde W} \cong \mathbb{P}(E)\setminus\{u_0 =0\}.$$  
  We denote by $D$  the divisor (image of the section) $\{u_0=0\}$ in  $\mathbb{P}(E)$  and by $ \pi$ the induced 
by $\pi_E$  the projection map $P(E)\to\tilde B.$


We have $\Aut(W)\subset \Aut({\tilde  W})$, since the (sub)set ${\tilde B} \subset  A$ is invariant under the action of $H_a.$ 
 On the other hand, according to \lemref{lem1}, $\Aut({\tilde W)}\subset \Aut(Y)$ where  $Y:=\mathbb{P}(E).$  
 Take any smooth projective closure  $\tilde X$ of $Y$ and   extend  $\pi$ to    the  rational  map $\tilde\pi:\tilde X\to A. $ 
Since $A$ contains no rational curves,  $\tilde \pi $ is a morphism, which is obviously projective.   Let $\tilde D$ be the closure of $ D$  in $\tilde X.$ Note that
  $\tilde D\cap Y=D,$
$\tilde\pi^{-1}({\tilde B} )=Y,$  and $\tilde p^{-1}(A\setminus{ \tilde B })=\tilde X\setminus Y$,
in light of the ``maximality" property of projective (and therefore proper) morphism $\pi $
 \cite[Section 2.6, pp. 95--96] {MuOd}).

Let $\tilde Z=\tilde X \setminus Y$. Let $\tilde \phi: \tilde{X}\dasharrow A \times \mathbb{P}^1$ be
 the rational extension of $\phi\circ  {\Psi}^{-1}: Y \dasharrow  A \times \mathbb{P}^1.$ 
Then $m(\tilde Z)=0,$ and the $A$-admissible triple  $(\tilde X, \tilde \phi , \tilde Z)$
 is the one we were looking for. \end{proof}

\begin{remark}\label{Kam}  This Lemma may be derived from  general results in \cite{KW} and \cite{Su1}, \cite{Su2}
 but  we prefer  an explicit construction which we use  in the next Lemma.\end{remark}

\begin{Lemma}\label{m2} Assume that  a triple  $(X, \phi,Z )$ is $A$-admissible,
 $m(Z)= 2$ and $r(Z)=2.$
 Then  there exists an $A$-admissible triple  $(\tilde X, \tilde \phi , \tilde Z)$
with $m(\tilde Z)=0$ and a group embedding $\Aut(W)\hookrightarrow \Aut (\tilde W),$
 where  $\tilde W:=\tilde X\setminus \tilde Z.$\end{Lemma}

\begin{proof} 
Since $Z_B$ contains two
 disjoint irreducible over $\mathcal K$ horizontal components  we may choose homogeneous coordinates  $(w_0: w_1)$ 
 in $\BP^1$   in such a way that  $\phi (Z_B)=B\times\{w_0 w_1=0\}.$  Thus this is the special case of \lemref{m1} 
when (in the notation of \lemref{m1})  $\tau_g=0$ whenever $t=0$
for all $g\in H_a.$
  Thus this lemma follows from \lemref{m1}.
 Note that in this case $\beta_{gh}\equiv  0$ and instead  of the $\BA^1$-bundle we have a line bundle.\end{proof}

It follows that the case $m(Z)\le 2$ may be reduced to the case $m(Z)=0.$

\begin{Lemma}\label{m00} If  a triple $(X, \phi,Z )$
is $A$-admissible,  $W:=X\setminus Z$
 and 
$m(Z)=0$ then there exists an $A$-admissible  triple $(\tilde X, \tilde\phi,\tilde Z )$

such that:

1) There is a group embedding   $\Aut (W)\hookrightarrow \Aut(\tilde W)$ where  $\tilde W=\tilde X\setminus \tilde Z;$

2)  If  $\tilde p $  is the projection map from $\tilde X$  onto $A$ induced by $\tilde \phi$, then $ \tilde Z=\tilde p^{-1}(S)$ for a certain closed subset $S$ of $A$; in addition,   for every point $b\in B:=A\setminus S$
the fiber $P_b=\tilde p^{-1}(b)$ is an irreducible reduced curve isomorphic to $\BP^1.$
\end{Lemma}

\begin{remark} In loose words it means that we can add to $Z$ all the singular fibers of $p$ 
without reducing an automorphism group. \end{remark}

\begin{proof} Since $m(Z)=0, $ we have $Z_A=p(Z) \ne A$ is a closed subset of $A.$  Let $a$  be a point of $Z_A$.    Then the fiber 
$P_a\cap W= P_a\setminus (Z\cap P_a)$ either has a non-projective irreducible component or is empty. 
Let $S_s$ be  the set of all points $a\in A\setminus Z_A$ such that   $P_a= p^{-1}(a)$
is singular (namely, has several irreducible components or a non-reduced  component).
Let $S:=S_s\cup  Z_A\subset A,$ i.e,  $ B:=A\setminus S$ is the set of all points $a\in A $ such that the fiber $P_a\subset W $  
 is a reduced irreducible smooth curve isomorphic to $\BP^1.$
Then   sets $B$ and $S$  are  invariant under the  action of  $H_a$  (see \eqref {seq1}), thus $\tilde W:= p^{-1}( B)$ is invariant under the action of  $\Aut(W),$ i.e, $\Aut(W)\subset \Aut(\tilde W).$

Thus, the  $A$-admissible triple  
$$\tilde X:=X, \tilde\phi:=\phi, \tilde Z:=p^{-1}(S)$$
enjoys the desired properties.
\end{proof}

\section{Proof of \thmref{main}}
\label{s3}

In this section we prove \thmref{main}  and \corref{main3d}.

\begin{proof}[Proof of Theorem \ref{main}]
By Remark \ref{smooth} we may assume that $W$ is smooth. 
When $W$ is projective, the desired result follows from \cite{MengZhang}. So, we may assume
 that quasiprojective $W$ is {\sl not} projective.  By Remark \ref{redtriple} we may choose such an     $A-$admissible 
$(X, \phi,Z )$  that $W=X\setminus Z.$
We use the exact sequence  \eqref{seq1}.

 It is proven in \cite [Section 6]{PS1} (see also \cite [Corollary 3.8]{BZ16}) that for an irreducible non-uniruled  variety $A$ the group 
$\Bir(A)$ (and, hence, $\Aut(A)$) is
 {\sl strongly} Jordan. 

  If $m(Z)>2$, then the (sub)group  $H_i$ in the short
   exact sequence \eqref{seq1}  is finite. 
If $m(Z)=2$, and $r(Z)=1,$  then, according to \lemref{field}, $H_i$ is bounded. 
It follows from Lemma \ref{sides}
 that in both cases $\Aut(W)$ is Jordan. (See also \cite[Lemma 2.8]{PS1}.)

According to \lemref{m2}, \lemref{m1}, \lemref{m00},  in all other cases  one  may  assume that conditions of \propref{m0} 
are satisfied, hence,  $\Aut(W)$ is strongly Jordan.


\end{proof}

\begin{proof}[Proof of \corref{main3d}]

Assume that $W$ is a quasiprojective irreducible variety of dimension $d\le 3.$ The case  $\dim W\le2$ was done in \cite{Popov1, ZarhinTG15,BZ15}. 
Assume that $ W$  is not  birational to
 $E\times \BP^2, $  where $E$ is an elliptic curve.  

If  $\dim(W) =3$  and $\Bir(W) $ is Jordan, then 
its subgroup $\Aut(W) $ is also Jordan.
 If $\Bir(W) $ is {\sl not} Jordan, then 
 according to \cite{PS2}, the variety $W$ has to be birational to $S\times \BP^1,$ where  $S$ is a surface
that enjoys one of the following  three properties.

{\bf Case  1.}     $S$  is an abelian surface. Since $S$ contains no rational curves, it  is rigid. Thus,$\Aut(W) $ is Jordan by \thmref{main}.

{\bf  Case  2.}   $S$ is bielliptic surface. Since $S$ contains no rational curves,  it is rigid. 
Thus,$\Aut(W) $ is Jordan by \thmref{main}.

 {\bf    Case  3.}  $S$ is a surface with Kodaira dimension  $ \varkappa(S)=1$ such that the 
Jacobian fibration of the pluricanonical fibration is locally trivial in Zariski topology.

Consider  {\bf    Case  3.}
Further on
 we   assume that $k=\BC$. 

  We have to prove that $S$ is rigid. Since
Jacobian fibration of the pluricanonical fibration is locally trivial in Zariski topology,   all fibers  (even the multiple ones)
of  the  pluricanonical fibration are smooth elliptic curves (\cite[Chapter VII, section 7,   Corollary  2]{Shaf}, \cite[Theorem 5.3.1]{CoDo}).

\begin{Lemma}\label{elhyp} Assume that $A$ is  a  smooth irreducible surface endowed with a 
morphism $\pi:A\to C$ such that 
\begin{itemize}\item $C$ is a smooth curve  of genus $g; $
\item Every fiber $F_c=\pi^{-1}(c), 
 c\in C$ is a smooth  elliptic curve;
\item Kodaira dimension $\varkappa (A)=1.$
\item  Morphism $\pi$ is a pluricanonical fibration, i.e  for some $N$ and every effective divisor $D\in|NK_A|$  there are positive numbers $\nu_1,\dots,\nu_n$ and fibers $F_1,\dots,F_n$  of $\pi$ such that $D=\sum\limits_{1}^{n}\nu_iF_i.$
\end{itemize}
 Then surface $A$   
contains no rational curves.\end{Lemma}

\begin{proof}
The surface $A$ enjoys the following properties:\begin{itemize}
\item  Euler characteristics $e(A)=0$ ( see \cite[Chapter IV, section 4, Theorm 6]{Shaf});
\item   Since $K_A^2=0,$ we have  $\chi(A)=  \chi(A, \mathcal O_A)=0  $  ( see  \cite{Shaf});
\item   If  fibration $\pi$ has precisely $k$ multiple fibers $F_1,\dots,F_k$ with multiplicities $m_1,\dots,m_k,$ respectively, then 
\begin{equation}\label{delta}
\delta(\pi):=2g-2+\sum\limits_{i=1}^{i=k}(1-\frac{1}{m_i})>0;\end{equation}
(see  \cite [Chapter V, proposition 12.5]{BHPV})
\item  In particular, $2g-2+k>0.$
\item Since $\pi$ is a  pluricanonical fibration  every automorphism $\phi\in \Aut(A)$ is  $\pi-$fiberwise. 
\item For every automorphism $\phi\in \Aut(A)$  the subset  $F_{sing}=F_1\cup\dots\cup F_k$ is invariant since multiple fibers go to multiple fibers;
\end{itemize}

Let $B\subset A$ be a rational curve. Since it cannot be contained in a fiber of $\pi,$ it is mapped by $\pi$ onto $C$
with some degree $m\ge1.$  Hence $C$ is rational.
Assume that $B$ intersects $F_i$ at points $a_1^i, \dots,a^i_{n_i}$  that are ramification points of restriction $\tilde \pi$  of $\pi$  onto $B,$ of orders 
 $r_{i,1}\dots,r_{i,n_i},$ respectively. Then \begin{itemize}
\item $r_{i,1}+\dots+r_{i,n_i}=m,$ 
\item $r_{i,j}\ge  m_i \ j=1,\dots , n_i;$
\item $n_i\le \frac{m}{m_i}.$
\end{itemize}

Assume that $\tilde \pi$  has also ramification points $b_1, \dots,b_r$  of orders $p_1,\dots,p_r$  respectively, (including nodes of $B$) outside  $ F_{sing}.$  

By the Hurwuitz formula we have 
\begin{equation}\label {Hurwitz}
2=2m-\sum\limits_{i=1}^{i=k}\sum\limits_{j=1}^{j=n_i}(r_{i,j}-1)-\sum\limits_{l=1}^{l=r}(p_l-1)=\end{equation}
$$2m-mk+\sum n_i  -\nu,$$
where $\nu:=\sum\limits_{l=1}^{l=r}(p_l-1)$ is a non-negative number. 

 Thus,  dividing by $m$ we get  $$k-2=\frac{1}{m}(-2+\sum n_i -\nu)\le \sum \frac{1}{m_i},$$
and $$\delta(\pi):=-2+\sum\limits_{i=1}^{i=k}(1-\frac{1}{m_i})=-2+k-\sum \frac{1}{m_i}\le 0$$
which contradicts to \eqref{delta}\end{proof}

Thus, in {\bf Case 3} surface $S$ is rigid as well, and  $\Aut(W) $ is Jordan by \thmref{main}.

\end{proof}

\begin{remark}
\label{complexC}
In the course of the proof of   \thmref{main}  and  \corref{main3d}  it suffices to consider the case when the ground field is the field $\BC$ 
of complex numbers. Indeed, suppose that we know that the Theorems  hold true when the ground field is $\BC$. Let $k$ be any algebraically closed field of characteristic $0$ and 
an algebraic variety $W$ over $k$ satisfies the conditions either of \thmref{main}  or \corref{main3d}.  Let us assume 
that $\Aut(W)$ is {\sl not} Jordan. We need to arrive to a contradiction.

The variety $W$  is defined over a subfield $k_0$ (of $k$) such  that $k_0$ is finitely generated over the field $\Q$ of rational numbers, i.e., there is
a quasiprojective variety $W_0$ over $k_0$ such that $W=W_0\times_{k_0}k$. (Clearly, $k_0$ is a countable field.)  Replacing if necessary $k_0$ by its finitely generated extension, we may assume that there is
 a   surface $A_0$ over $k_0$ and a $k_0$-birational map between $W$ and $A_0\times \mathbb{P}^1$.  Moreover,  we may choose $k_0$ in such  a way that 
\begin{itemize}\item
 if $A$ is bielliptic,  the same is valid for $A_0$  (the bielliptic structure would be defined over $k_0$);
\item  if $\varkappa(A)=1,$ the same is valid for $A_0$  (the pluricanonical fibration  would be defined over $k_0$);
\item  if a pluricanonical fibration of $A$ has smooth irreducible elliptic fibers, the same is valid for $A_0$ (smoothness is preserved under base change  \cite[Proposition3.38, Chapter 4]{Liu})
\item if $A$ contains no rational curves the same is valid for  $A_0.$   Indeed, if $A_0$ contained  a rational  curve, then
for some integer $d$  one  of the irreducible  quasiprojective components of the variety $RatCurves ^n_d(A)$  (see \cite[Definition 2.11]{Kollar})  would have  a point over $\BC.$  But then it would have  a point over $k$ as well, since $k$ is algebraically closed. \end{itemize}


 The non-Jordanness of   $\Aut(W)$ means that there exists an infinite sequence of finite subgroups $\{G_i \subset \Aut(W)\}_{i=1}^{\infty}$, whose Jordan indices $J_{G_i}$ tend to infinity.  For each positive $i$ there is a subfield $k_i$ of $k$ that contains $k_0$ and is finitely generated over $k_0$, and such that all automorphisms from $G_i$ are defined over $k_i$. Clearly,  all $k_i$ are countable fields.
The compositum of all $k_i$'s (in $k$) is countably generated over $k_0$ and therefore is also a countable field.
  Let us consider the algebraic closure $k_{\infty}$ of this compositum in $k$. Clearly,
$k_{\infty}$ is an algebraically closed countable subfield of $k$ that contains all $k_i$. Let us consider the quasiprojective variety $W_{\infty}=W_0\times_{k_0}k_{\infty}$. Clearly, there exist group embeddings $G_i \hookrightarrow \Aut_{k_{\infty}}(W_{\infty})$ for all  positive $i$.  This implies that $\Aut_{k_{\infty}}(W_{\infty})$ is {\sl not} Jordan.

Since 
$k_{\infty}$  is countable,  there is a field embedding $k_{\infty}\hookrightarrow \BC$.
 Let us consider the complex quasiprojective variety $W_{\BC}=W_{\infty}\times_{k_{\infty}}\BC$, which is birational to $A_{\BC}\times \BP^1$ where $A_{\BC}=A\times_{k_0}\BC$ is a complex variety meeting conditions of \thmref{main}.
In particular, $\Aut(W_{\BC})$ is Jordan. On the other hand,
  there is a group embedding $\Aut(W_{\infty})\hookrightarrow \Aut(W_{\BC})$. This implies that $\Aut(W_{\BC})$ is {\sl not} Jordan as well, which gives us the desired contradiction.
\end{remark}


\begin{thebibliography}{9999}


\bibitem[BHPV ]{BHPV} W. Barth, K. Hulek,   C. Peters,  A. van de Ven, {\sl Compact Complex Surfaces}.
Springer-Verlag, Berlin, 2004.


\bibitem[BZ1]{BZ15} T. Bandman, Yu. G. Zarhin,   {\sl Jordan groups and algebraic surfaces}.
Transformation Groups {\bf 20} (2015), 327--334. 


\bibitem[BZ2]{BZ16} T.Bandman, Yu. G. Zarhin,  {\sl Jordan groups, conic bundles and abelian varieties} .
 Algebraic Geometry {\bf 4:2} (2017), 229--246.

\bibitem[Bir]   {Bir} C. Birkar, {\sl  Singularities of linear systems and boundedness of Fano varieties}. 2016,
arXiv:1609.05543.


\bibitem[BM]{B-M}  E. Bombieri, D. Mumford, {\sl Enriques'   Classification of surfaces in Char.  $p$,   II}.
In: Complex Analysis   and Algebraic Geometry  (W.L.Baily Jr.,    T. Shioda,  eds.),  Cambridge Univ. Press, 1977,  pp. 23-43.

\bibitem[C]{C}   B. Conrad, {\sl A Modern Proof of the Chevalley' s Theorem on Algebraic Groups}.
J. Ramunajam Math. Soc., {\bf 17} (2002),  no. 1, 1--18.  

\bibitem[C-D]{CoDo}  F.  Cossec,    I. Dolgachev,    Enriques surfaces I. Birkhauser,  Berlin,  1989.



\bibitem[CR]{CR} C.W. Curtis, I. Reiner, Representation Theory of Finite Groups and Associative Algebras. Wiley, New York, 1962.




\bibitem[De]{Deba} O. Debarre, Higher-Dimensional Algebraic Geometry,   Springer-Verlag, New York, 2001.






\bibitem[Gr1]{Gro1}  A.  Grothendieck, {\sl  Technique de construction  et th\'eor\`mes d'existence en g\'eom\'etrie alg\'ebrique IV:
les sch\'emas de Hilbert}, S\'eminaire N. Bourbaki, 1960-1961, exp. 221,   249-276.






\bibitem[EGA]{EGA43}  A. Grothendieck, {\sl \'El\'ements de g\'eom\'etrie alg\'ebrique (r\'edig\'es avec la collaboration de J. Dieudonn\'e): 
IV, \'Etude locale des sch\'emas et des morphismes de sch\'emas, Troisi\'eme  partie}. Publ.  Math. IHES {\bf 28} (1966).


\bibitem[It]  
{Itaka} Sh.  Iitaka, Algebraic Geometry, GTM {\bf 76}.  Springer-Verlag, Berlin Heidelberg New York, 1982


\bibitem[KW]{KW} T. Kambayashi, D. Wright,  {\sl Flat families of affine lines are affine-line bundles}. Illinois J. Math. {\bf 29} (1985), no. 4, 672--681. 

\bibitem[Ko]  {Kollar} J. Kollar, Rational curves on algebraic varieties.  Ergebnisse der Math. 3 Folge {\bf 32}, Springer-Verlag, Berlin Heidelberg New York, 1996.


\bibitem [La]{Lang}   S. Lang, Algebra, 2nd edition. Addison-Wesley, Reading, MA, 1993.

\bibitem [Li]{Liu}   Q. Liu, Algebraic Geometry ang Arithmetic Curves. Oxford graduate texts in Mathematics, {\bf 6}, New York, 2002.

\bibitem[Mat]{Mat}  T.   Matsusaka, {\sl Polarized Varieries, Fields of Moduli and generalized Kummer Varieties of Polarized varieties}.  American J. Math.,
{\bf 80}, no. 1, 45--82.

\bibitem[MO]{MuOd}   D. Mumford,   T. Oda,   Algebraic Geometry II. Texts and Reading in Mathematics, 73, Hindustan Book Agency, Mumbai, 2015.

\bibitem[MZ]{MengZhang} Sh. Meng, D.-Q. Zhang, {\sl Jordan property for non-linear algebraic groups and projective varieties}.  American J. Math., to appear; arXiv: 1507.02230 [math.AG] .





\bibitem[Mu1]{MumfordRed} D. Mumford,  The Red Book of Varieties and Schemes. Lecture Notes in Math.
  vol. {\bf 1358}, Springer, 1999.

\bibitem[Mu2]{MumfordAb} D. Mumford, Abelian Varieties, 3rd edition. Hindustan Book Agency, India, Mumbai, 2008.
  

\bibitem[MuTu]{MuTu}  I. Mundet i Riera, A.Turull, {\sl Boosting An Analogue of Jordan's Theorem For Finite Groups}.  Adv. Math. {\bf 272} (2015), 820--836.  



\bibitem[Po1]{Popov1} V.L. Popov, {\sl On the Makar-Limanov, Derksen invariants, and finite automorphism groups of algebraic varieties}, pp.  289--311. In:  Affine algebraic geometry: the Russell Festschrift.  CRM Proceedings and Lecture Notes {\bf 54}, Amer. Math. Soc., 2011.


\bibitem[Po2]{Popov2} V.L. Popov, {\em Jordan groups and automorphism groups of algebraic varieties}. In: Automorphisms in Birational and Affine Geometry,
Springer Proceedings in Mathematics and Statistics {\bf 79} (2014), 185--213.


\bibitem[PS0]{PS0} Yu. Prokhorov and C. Shramov, {\sl Jordan property for Cremona groups}. Amer. J. Math. {\bf 138} (2016), no. 2, 403--418.

\bibitem[PS1]{PS1} Yu. Prokhorov and C. Shramov, {\sl Jordan Property for groups of birational Selfmaps}. Compositio Math. {\bf 150} (2014), 2054--2072.

\bibitem[PS2]{PS2} 
Yu. Prokhorov, C. Shramov, 
{\sl Finite groups of birational selfmaps of threefolds}. arXiv:1611.00789 . 


\bibitem[Ros]{Ros}  M. Rosenlicht, {\sl  A Remark on Quotient Spaces}. An. Acad. Brasil Ci, {\bf 35} (1963), 487--489.




\bibitem[Sa]{Sa}  F. Sakai,      {\sl Kodaira Dimension of Complements of Divisors},  pp.239-259.
In: Complex Analysis   and Algebraic Geometry  (W.L.Baily Jr.,    T. Shioda, eds.),  Cambridge University Press, 1977. 


\bibitem[Se]{Se} F. Serrano, {\sl Divisors of Bielliptic surfaces and Embedding in $\BP^4$}.  Math.  Z. {\bf 203}  (1990), 527--533.
 

\bibitem[Ser]{SerreFG} J.-P. Serre, Finite groups: an introduction. International Press, Somerville, MA, 2016.

\bibitem[Sh]{Shaf} I.R. Shafarevich et al., Algebraic Surfaces. Proc. Steklov Inst. Math. {\bf 75}, Moscow, 1965; American Mathematical Society,  Providence, RI, 1967.

\bibitem[Su1]{Su1}  H. Sumihiro,  {\sl Equivariant completion}. J. Math. Kyoto Univ. {\bf 14 } (1974), 1--28.


\bibitem[Su2]{Su2}  H. Sumihiro,   {\sl Equivariant completion}. II. J. Math. Kyoto Univ. {\bf 15} (1975), no. 3, 573--605. 

\bibitem[Za1]{ZarhinPEMS} Yu.G. Zarhin,  {\sl Theta groups and products of  abelian and rational varieties}. Proc.
 Edinburgh Math. Soc. {\bf 57:1} (2014), 299--304.
 
\bibitem[Za2]{ZarhinTG15}  Yu.G. Zarhin,  {\sl Jordan groups and elliptic ruled surfaces}. Transformation Groups  {\bf 20} (2015), no. 2, 557--572.





\end{thebibliography}
\end{document}